\documentclass[11pt]{article}

\usepackage{amsmath}
\usepackage{amssymb}
\usepackage{amsfonts}
\usepackage{amsthm}
\usepackage{array}
\usepackage{bm}
\usepackage{enumitem}

\newcolumntype{C}[1]{>{\centering\hspace{0pt}}p{#1}}
\usepackage[pdftex]{graphicx}
\usepackage[margin=1in]{geometry}

\newtheorem{thm}{Theorem}[section]

\newtheorem{lem}[thm]{Lemma}

\theoremstyle{definition}

\theoremstyle{definition}
\newtheorem*{rmk}{Remark}

\usepackage[nottoc]{tocbibind}
\numberwithin{equation}{section}

\usepackage[hidelinks]{hyperref}

\title{Free-Boundary Minimal Surfaces of Constant \\ K\"{a}hler Angle in Complex Space Forms}
\author{Jesse Madnick}
\date{November 2020}

\newcommand{\Addresses}
{{  \bigskip
  \textsc{National Center for Theoretical Sciences} \par\nopagebreak
  \textsc{National Taiwan University}\par\nopagebreak
  \textsc{Taipei, Taiwan}\par\nopagebreak
 \textit{E-mail address}: \texttt{jmadnick@ncts.ntu.edu.tw}
}}

\begin{document}

\maketitle

\begin{abstract}
% In real space forms, Fraser and Schoen proved that a free-boundary minimal disk in a geodesic ball is totally geodesic.  It is natural to ask whether a similar uniqueness statement is true in complex space forms. \\
%\indent In this note, we show that in geodesic balls of $\mathbb{CP}^2$ and $\mathbb{CH}^2$, a free-boundary minimal disk with constant K\"{a}hler angle is totally geodesic.  Slightly more generally, we show that a free-boundary totally-real minimal disk in geodesic balls of $\mathbb{CP}^n$ or $\mathbb{CH}^n$ must be superminimal.  \\

In real space forms, Fraser and Schoen proved that a free-boundary minimal disk in a geodesic ball is totally geodesic.  In this note, we consider free-boundary minimal surfaces $\Sigma$ (of any genus) in geodesic balls of complex space forms.  In $\mathbb{CP}^2$, $\mathbb{C}^2$ and $\mathbb{CH}^2$, we show that if $\Sigma$ is Lagrangian, then $\Sigma$ is totally geodesic.  In $\mathbb{CP}^n$, $\mathbb{C}^n$ and $\mathbb{CH}^n$ for $n \geq 2$, we show that if $\Sigma$ has K\"{a}hler angle $\pi/2$, then $\Sigma$ is superminimal.

\end{abstract}

% \tableofcontents
% \pagebreak
 
\section{Introduction}

\indent \indent In a Riemannian manifold $M$ with boundary, a \textit{free-boundary minimal surface} is a minimal surface $u \colon \Sigma^2 \to M$ with $u(\partial \Sigma) \subset \partial M$ such that $u(\Sigma)$ meets $\partial M$ orthogonally.  Interest in the orthogonality condition comes from the first variation of area.  Indeed, if $u_t \colon \Sigma \to M$ with $u_t(\partial \Sigma) \subset \partial M$ is a one-parameter family of immersions with $u_0 = u$, then
$$\left. \frac{d}{dt}\right|_{t = 0} \text{Area}(u_t(\Sigma)) = -\int_\Sigma \langle H, X \rangle\,\text{vol}_\Sigma + \int_{\partial \Sigma} \langle \mu, X \rangle\,\text{vol}_{\partial \Sigma}$$
where $H$ is the mean curvature of $u(\Sigma)$, where $X$ is the variation vector field, and where $\mu$ is the unit vector field in $T\Sigma$ that is orthogonal to $T(\partial \Sigma)$ and outward-pointing.  This illustrates that $\left. \frac{d}{dt}\right|_{t = 0} \text{Area}(u_t(\Sigma)) = 0$ for all variations of $u$ if and only if $u(\Sigma)$ is a free-boundary minimal surface. For an excellent recent survey, see \cite{li2019free}.

\indent Generalizing results of Nitsche \cite{nitsche1985stationary} and Souam \cite{souam1997stability}, Fraser and Schoen \cite{fraser2015uniqueness} proved that a free-boundary minimal disk in a geodesic ball in a real space form is totally geodesic.  In this note, we ask whether an analogous uniqueness statement holds in complex space forms.  We show:  %In place of an assumption on the topology of $\Sigma$, we instead consider a restriction on the K\"{a}hler angle.  We show:

\begin{thm} \label{thm:Main}
Let $u \colon \Sigma^2 \to B^{2n}$ be a compact orientable free-boundary minimal surface in a geodesic ball in a complex space form of real dimension $2n$. \\ % If $u(\Sigma)$ has constant K\"{a}hler angle, then the (possibly disconnected) boundary curve $u(\partial\Sigma)$ is a geodesic in $\partial B$.  Moreover: \\
\indent (a) Suppose $n \geq 2$.  If $u(\Sigma)$ has K\"{a}hler angle $\frac{\pi}{2}$ or $\frac{3\pi}{2}$, then $u(\Sigma)$ is superminimal. \\
\indent (b) Suppose $n = 2$.  If $u(\Sigma)$ is Lagrangian, then $u(\Sigma)$ is totally geodesic.
\end{thm}

We emphasize that the hypotheses of Theorem \ref{thm:Main} involve no assumption on the topology of $\Sigma$ beyond the basic requirement that $\Sigma$ be a compact orientable surface with boundary.  In particular, we make no \textit{a priori} assumption on the genus or number of boundary components of $\Sigma$. \\ %or the number of boundary components of $\Sigma$.  Our only requirement on the topology is the standing assumption that $\Sigma$ is a compact smooth surface with non-empty boundary. \\

\indent Our proof will be complex-analytic, similar in spirit to Fraser and Schoen's arguments in \cite{fraser2015uniqueness}.  Now, in \cite{fraser2015uniqueness}, the constant sectional curvature assumption in used in two places.  First, thanks to the Codazzi equation, minimal surfaces in real space forms admit a natural holomorphic quartic form $Q$.  Second, to demonstrate that $Q$ vanishes, Fraser and Schoen use the fact that geodesic spheres in real space forms are totally umbilic. \\
\indent In our situation, by contrast, the complex space forms $\mathbb{CP}^n$ and $\mathbb{CH}^n$ admit no totally-umbilic hypersurfaces whatsoever.  To compensate for this, we will instead exploit the fact that geodesic spheres $S$ in complex space forms are \textit{Hopf hypersurfaces}, by which we mean that the complex structure endomorphism applied to a unit normal vector to $S$ is principal. \\
\indent Now, in place of the holomorphic quartic form $Q$, we analyze a certain holomorphic \textit{cubic} form $P$ introduced in the 1983 papers of Eells and Wood \cite{eells1983harmonic} and Chern and Wolfson \cite{chern1983minimal}.  The cubic form $P$ has since been used in several studies of harmonic maps and minimal surfaces: see, e.g., \cite{wolfson1985minimal}, \cite{wolfson1989minimal}, \cite{eells1985twistorial}, \cite{ma2005totally}, \cite{castro1997twistor}, \cite{castro1995twistor}. \\
\indent We will define $P$ precisely in (\ref{eq:HoloCubicForm}).  For now, note that a minimal surface is called \textit{superminimal} if $P = 0$ on the surface.  In $\mathbb{CP}^2$, there exists a great variety of compact superminimal surfaces \cite{eells1985twistorial}, \cite{chern1983minimal}.  In $\mathbb{CP}^n$, every superminimal surface can be constructed from holomorphic curves \cite{wolfson1985minimal}, which explains the significance of Theorem 1.1(a).

\indent Theorem 1.1(b) is interesting in view of the abundance of Lagrangian minimal surfaces in $\mathbb{CP}^2$, $\mathbb{C}^2$ and $\mathbb{CH}^2$.  Heuristically, the rigidity we observe can be explained as follows.  In a K\"{a}hler $4$-manifold, a minimal Lagrangian $u \colon \Sigma \to B^4$ has only two independent component functions in its second fundamental form.  Along $\partial \Sigma$, the free-boundary condition together with the Hopfness of the geodesic sphere $\partial B$ imposes two constraints on these two functions, which forces the second fundamental form of $u(\Sigma)$ to vanish along $\partial \Sigma$.   %This heuristic count will be made precise in the proofs below.

\begin{rmk}
In $\mathbb{CH}^2$, it is likely that Theorem 1.1(b) is still true if ``geodesic ball" is replaced by ``horoball" --- the domain whose boundary is the other Hopf hypersurface in $\mathbb{CH}^2$ with exactly two distinct constant principal curvatures \cite{montiel1985real} --- but I have not checked the details.
\end{rmk}

\begin{rmk} After this work was completed, I learned of the recent paper \cite{li2020lagrangian} of Mingyang Li, Guofang Wang, and Liangjun Weng.  As we now explain, our Theorem 1.1(b) is similar to the main result of \cite{li2020lagrangian}, although neither implies the other.  Moreover, both works center on showing the vanishing of the holomorphic cubic form $P$. \\
\indent The main result of \cite{li2020lagrangian} asserts that a Lagrangian minimal \textit{disk} in a geodesic ball of $M^4 = \mathbb{C}^2$ with \textit{Legendrian boundary} and \textit{constant contact angle} is totally geodesic.  In the context of Lagrangians in $4$-manifolds, the free-boundary condition is a special case of the more general ``Legendrian boundary and constant contact angle" condition.  Thus, compared with our Theorem 1.1(b), the rigidity result of \cite{li2020lagrangian} entails a stricter topological assumption, but a more general boundary condition.  The authors in \cite{li2020lagrangian} also note that their result is still true if the ambient manifold is any complex space form, though the details are left to the reader. \\
%\indent In \cite{li2020lagrangian}, the authors show that a Lagrangian minimal \textit{disk} in a geodesic ball of $M^4 = \mathbb{C}^2$ with \textit{Legendrian boundary} and \textit{constant contact angle} is totally geodesic.  In the Lagrangian setting, the free-boundary condition is a special case of the ``Legendrian boundary and constant contact angle" condition.  Thus, Theorem 1.1 of \cite{fraser2015uniqueness} % Hence, in the case of $M^4 = \mathbb{C}^2$ and $\Sigma$ homeomorphic to a disk, their result implies our Theorem 1.1(b).  (Of course, that special case also follows from \cite{fraser2015uniqueness}.)  Moreover, the authors note that their result is still true if the ambient manifold is any complex space form, though the details are left to the reader.  \\
\indent Their article also contains interesting examples of Lagrangians in $\mathbb{C}^2$ satisfying the Legendrian boundary condition.  Their examples do not include a minimal Lagrangian annulus satisfying the \textit{free}-boundary condition, leading the authors to conjecture that no such annulus exists.  Our Theorem 1.1(b) establishes this conjecture.
\end{rmk}

\noindent \textbf{Acknowledgements:} I thank David Wiygul for teaching me about free-boundary minimal surfaces, and thank Pat Ryan for sharing with me his beautiful book on hypersurfaces with Thomas Cecil \cite{cecil2015geometry}, which greatly aided in the preparation of this work.  I thank Rick Schoen, Wei-Bo Su, and Chung-Jun Tsai for clarifying conversations that led to a strengthening of Theorem 1.1, and Gavin Ball, Da Rong Cheng, Jih-Hsin Cheng, and Spiro Karigiannis for their interest and encouragement.

Part of this work was completed during the author's postdoctoral fellowship at the National Center for Theoretical Sciences (NCTS) at National Taiwan University; I thank the Center for their support.

\section{Proof of Main Result}

\indent \indent Let $M$ be a complex space form of real dimension $2n$, so that $M$ is $\mathbb{CP}^n, \mathbb{C}^n$, or $\mathbb{CH}^n$ equipped with a metric $\langle \cdot, \cdot \rangle$ of constant holomorphic sectional curvature.  Let $\overline{\nabla}$ denote the Levi-Civita connection of $\langle \cdot, \cdot \rangle$, let $J$ denote the ($\overline{\nabla}$-parallel) complex structure on $M$, and let $\Omega( \cdot, \cdot) = \langle J \cdot, \cdot \rangle$ denote the K\"{a}hler form on $M$. \\

\indent Let $B$ denote a geodesic ball in $M$, and let $S = \partial B$ denote its boundary sphere.  Let $\nu$ denote the outward-pointing unit normal vector field to $S$.  Let $A \colon TS \to TS$ denote the shape operator of $S$, by which we mean
%\indent  [Almost contact metric structure] [Define $-J\nu$] \\
$$A(X) = \overline{\nabla}_X \nu$$
We emphasize that $S$ is \textit{not} totally-umbilic.  Indeed, $S$ has two distinct (constant) principal curvatures \cite{cecil2015geometry}, one of multiplicity 1, and one of multiplicity $(2n-2)$.  Moreover, $S$ is a \textit{Hopf hypersurface}, meaning that the (Reeb) vector field $-J\nu$ is principal \cite{cecil2015geometry}.  We denote the (multiplicity $1$) principal curvature of $-J\nu$ by $a$ and the multiplicity $(2n-2)$ principal curvature by $\lambda$, so that
\begin{align*}
A(J\nu) & = aJ\nu \\
A(V) & = \lambda V, \ \ \text{ for all }V \in TS \text{ with } V \perp J\nu.
\end{align*}
For more on geodesic spheres in complex space forms, the reader might consult \cite{cecil2015geometry}, \cite{niebergall1997real}, \cite{maeda2000geometry}, \cite{maeda2018geometry}. \\

\indent Let $u \colon \Sigma \to B$ be a compact orientable free-boundary minimal surface, equip $\Sigma$ with an orientation, and let $\theta$ denote the K\"{a}hler angle of the immersion.  The bundle of vector fields along $u(\Sigma)$ decomposes as $u^*(TM) = u_*(T\Sigma) \oplus N\Sigma$, and we denote the second fundamental form of $u(\Sigma)$ as
$$\text{I\!I}(X,Y) = (\overline{\nabla}_XY)^{N\Sigma}$$
where the superscript $N\Sigma$ denotes the projection $u^*(TM) \to N\Sigma$. \\ % We equip $\Sigma$ with an orientation. \\

\indent For calculations, we now define a local frame that is adapted to the geometry of $u \colon \Sigma \to B$.  To begin, let $(e_1, e_2)$ be a local oriented orthonormal frame defined in a neighborhood $W$ of a point on $\partial \Sigma$ such that $\nu = u_*(e_1)$ along $\partial \Sigma$.  Extend $\nu$ to a vector field on $W$ by requiring
$$\nu = u_*(e_1)$$
and set
\begin{align*}
T & = u_*(e_2).
\end{align*}
At points $p \in W$, let $\mathcal{D} = \text{span}(\nu, J\nu)^\perp$ denote the ($J$-invariant) real $(2n-2)$-plane field orthogonal to the real $2$-plane field $\text{span}(\nu, J\nu)$.  So, both $\mathcal{D}_p$ and $N_p\Sigma$ are $(2n-2)$-planes inside the $(2n-1)$-plane $\text{span}(\nu)^\perp$.  At points where $\sin(\theta) \neq 0$, the intersection $N_p\Sigma \cap \mathcal{D}_p$ is a $(2n-3)$-plane.  However, at points where $\sin(\theta) = 0$, we have $N_p\Sigma = \mathcal{D}_p$.

\begin{rmk} If $u$ is minimal and not $\pm$-holomorphic, the set of points at which $\sin(\theta) = 0$ is discrete.  See \cite{chern1983minimal}: $\S$2.
\end{rmk}

\indent Let $\{U, JU, V_1, \ldots, V_{2n-4}\}$ be a unitary basis for $\mathcal{D}$ with the property that
$$U, V_1, \ldots, V_{2n-4} \in N\Sigma \cap \mathcal{D}.$$
Thus, $(\nu, J\nu, U, JU, V_1, \ldots, V_{2n-4})$ is a local unitary frame along $u(\Sigma)$.  In terms of this frame, we can write $T = c_1 J\nu + c_2JU$ for some functions $c_1, c_2$ satisfying $(c_1)^2 + (c_2)^2 = 1$.  Since $\cos(\theta) = \Omega(\nu, T) = \Omega(\nu, c_1 J\nu + c_2 JU) = c_1$, it follows that $c_2 = \pm \sin(\theta)$.  Now, $U$ has only been specified up to sign: we choose the sign such that $c_2 = -\sin(\theta)$.  Thus,
$$T = \cos(\theta) J\nu - \sin(\theta) JU.$$
Finally, let $N$ denote the vector field
$$N = -\sin(\theta) J\nu - \cos(\theta) JU.$$
One can check that $\{U, N, V_1, \ldots, V_{2n-4}\}$ is an orthonormal basis of each normal space $N_p\Sigma$.  The upshot is that
\begin{equation} \label{eq:frame}
(\nu, T, U, N, V_1, \ldots, V_{2n-4})
\end{equation}
is a local orthonormal frame adapted to the free-boundary surface $u \colon \Sigma \to B$. \\

%\pagebreak

\indent We now express the second fundamental form of $u(\Sigma)$ in terms of the frame (\ref{eq:frame}), writing
\begin{align*}
\text{I\!I}(e_1, e_1) & = \textstyle a_{11} U + b_{11} N + \sum h^\lambda_{11} V_\lambda \\ 
\text{I\!I}(e_1, e_2) & = \textstyle a_{12} U + b_{12} N + \sum h^\lambda_{12} V_\lambda \\
\text{I\!I}(e_2, e_2) & = -\text{I\!I}(e_1, e_1)
%\text{I\!I}(e_2, e_2) & = \textstyle \sum h^\lambda_{22} V_\lambda + h^U_{22} U + h^N_{22} N
\end{align*}
where $a_{11}, a_{12}, b_{11}, b_{12}$ and $h^\lambda_{11}, h^\lambda_{12}$ are functions, and $1 \leq \lambda \leq 2n-4$.  In this notation, we consider the cubic form $P$ given by
\begin{equation} \label{eq:HoloCubicForm}
P = \frac{1}{4}\sin(\theta) \left[ ( a_{11} - b_{12} ) - i( a_{12} + b_{11} ) \right] \phi^3
\end{equation}
where $\phi = \epsilon_1 + i\epsilon_2 \in \Omega^{1,0}(\Sigma)$, and $(\epsilon_1, \epsilon_2)$ is the coframe field dual to $(e_1, e_2)$.  In Theorem 2 of \cite{chern1983minimal}, Chern and Wolfson show that if $u(\Sigma)$ is a minimal surface in a complex space form, then $P$ is holomorphic. \\

\indent We can now establish two lemmas.  The first is essentially a rephrasing of equation (2.30) in \cite{chern1983minimal}, which we prove here for the sake of being self-contained.  It shows, in particular, that minimal surfaces of constant K\"{a}hler angle have extra symmetries in their second fundamental forms.

\begin{lem} \label{lem:KahlerAngle}
For any tangent vector $X \in T\Sigma$, we have:
$$d\theta(X) = \langle \mathrm{I\!I}(X,e_2), N \rangle + \langle \mathrm{I\!I}(X,e_1), U \rangle $$
In particular,
\begin{align*}
d\theta(e_1) & = a_{11} + b_{12} \\
d\theta(e_2) & = a_{12} - b_{11} \\
\end{align*}
\end{lem}
\begin{proof} By differentiating $\langle T, J\nu \rangle = \cos(\theta)$, we find that
\begin{align*}
-\sin(\theta) d\theta(X) = \overline{\nabla}_X(\cos(\theta)) & = \overline{\nabla}_X \langle T, J\nu \rangle \\
& = \langle \overline{\nabla}_XT, J\nu\rangle + \langle \overline{\nabla}_X(J\nu), T  \rangle \\
& = \langle \overline{\nabla}_XT, J\nu\rangle - \langle \overline{\nabla}_X\nu, JT  \rangle \\
& = \langle \overline{\nabla}_XT, \cos(\theta)T - \sin(\theta) N \rangle - \langle \overline{\nabla}_X \nu, -\cos(\theta)\nu + \sin(\theta)U \rangle \\
& = -\sin(\theta) \langle \overline{\nabla}_XT, N \rangle - \sin(\theta)\langle \overline{\nabla}_X\nu, U \rangle.
\end{align*}
Thus,
$$\sin(\theta) d\theta(X) = \sin(\theta) \left[ \langle \text{I\!I}(X,e_2), N \rangle + \langle \text{I\!I}(X,e_1), U \rangle  \right]\!.$$
This establishes the claim at points where $\sin(\theta) \neq 0$.  By a completely analogous calculation, differentiating $\langle T, JU \rangle = -\sin(\theta)$ yields
%\begin{align*}
%\cos(\theta) d\theta(X) = \overline{\nabla}_X(\sin(\theta)) & = \overline{\nabla}_X\langle T, JU \rangle \\
%& = \langle \overline{\nabla}_X T, JU \rangle + \langle \overline{\nabla}_X(JU), T  \rangle \\
%& = \langle \overline{\nabla}_XT, JU \rangle - \langle \overline{\nabla}_XU, JT  \rangle \\
%& = \langle \overline{\nabla}_XT, \sin(\theta)T + \cos(\theta)N \rangle - \langle \cos(\theta)\nu - \sin(\theta)U, \overline{\nabla}_XU \rangle  \\
%& = \cos(\theta) \langle \overline{\nabla}_XT, N \rangle - \cos(\theta) \langle \overline{\nabla}_XU, \nu \rangle
%\end{align*}
%Thus,
$$\cos(\theta)d\theta(X) = \cos(\theta) \left[ \langle \text{I\!I}(X, e_2), N \rangle +  \langle \text{I\!I}(X,e_1), U \rangle \right]\!,$$
which establishes the claim at points where $\cos(\theta) \neq 0$.
\end{proof}

We now exploit the free-boundary condition and the Hopfness of $\partial B$.  The following quick calculation is the analogue of equation (2.5) in \cite{fraser2015uniqueness}.

\begin{lem} \label{lem:ShapeOp}
 Along $\partial \Sigma$, we have
$$\mathrm{I\!I}(e_1, e_2) = (\lambda-a) \cos(\theta) \sin(\theta) N.$$
\end{lem}
\begin{proof}
We compute
\begin{align*}
A(T) & = A( \cos(\theta) J\nu - \sin(\theta) JU) \\
& = a \cos(\theta) J\nu - \lambda \sin(\theta) JU \\
& = a \cos(\theta) \left(\cos(\theta)T - \sin(\theta)N \right) + \lambda \sin(\theta) \left(\sin(\theta)T + \cos(\theta) N \right) \\
& = \left(a \cos^2(\theta) + \lambda \sin^2(\theta) \right)T + (\lambda-a) \cos(\theta) \sin(\theta)N.
\end{align*}
Consequently,
$$\text{I\!I}(e_1, e_2) = (\overline{\nabla}_T \nu)^{N\Sigma} = (A(T))^{N\Sigma} = (\lambda-a) \cos(\theta) \sin(\theta) N.$$
\end{proof}

 We now prove Theorem \ref{thm:Main}.

 \begin{proof} (a) Let $u \colon \Sigma^2 \to B^{2n}$ be a free-boundary minimal surface in a geodesic ball $B$, where $n \geq 2$.  Suppose that $u(\Sigma)$ has K\"{a}hler angle $\theta = \frac{\pi}{2}$ or $\theta = \frac{3\pi}{2}$, so $\cos(\theta) = 0$.  Since $d\theta = 0$, Lemma \ref{lem:KahlerAngle} gives
\begin{align*}
a_{11} + b_{12} & = 0 & a_{12} - b_{11} & = 0
\end{align*}
on all of $\Sigma$.  Now, Lemma \ref{lem:ShapeOp} shows that $a_{12} = b_{12} = 0$ along $\partial \Sigma$, so that $a_{11} = b_{11} = 0$ along $\partial \Sigma$ as well, and hence $P = 0$ along $\partial \Sigma$. \\
\indent Let $q \in \partial \Sigma$, and let $(z)$ denote a complex coordinate on a neighborhood $E$ of $q$, so that on $E$ we have $P = f(z)\,dz^3$ for some holomorphic function $f \colon E \to \mathbb{C}$.  Since $P = 0$ along $\partial \Sigma$, we have $f = 0$ on $E \cap \partial \Sigma$, so that the holomorphicity of $f$ implies that $f = 0$ on all of $E$.  Thus, $P$ vanishes on an open set of $\Sigma$, so (since $P$ is holomorphic) $P = 0$ on all of $\Sigma$, meaning that $u(\Sigma)$ is superminimal. \\
%  Since $P$ is holomorphic, it follows that $P = 0$ on all of $\Sigma$,

\indent (b) Suppose now that $n = 2$.  By part (a), we know that $u(\Sigma)$ is superminimal.  By Lemma \ref{lem:KahlerAngle} and (\ref{eq:HoloCubicForm}), every superminimal Lagrangian in $M^4$ is totally geodesic.  This completes the proof.
  \end{proof}

  \begin{rmk} In a complex space form of real dimension $4$, a minimal surface $u(\Sigma)$ of \textit{constant} K\"{a}hler angle $\theta$ must satisfy $\cos(\theta) = 0$ or $\sin(\theta) = 0$ (e.g., by equation (2.32) of \cite{chern1983minimal}), meaning that $u(\Sigma)$ is either Lagrangian or $\pm$-holomorphic.

  \indent We remark that if $u \colon \Sigma^2 \to B^4$ is a free-boundary minimal surface that is either Lagrangian or $\pm$-holomorphic, then its boundary $u(\partial \Sigma)$ is a geodesic in $\partial B$.  To see this, note that Lemma \ref{lem:ShapeOp} implies that $\text{I\!I}(e_1, e_2) = 0$ along $\partial \Sigma$, so $a_{12} = b_{12} = 0$ along $\partial \Sigma$.  Lemma \ref{lem:KahlerAngle} then implies $a_{11} = b_{11} = 0$, and hence (since $\dim_{\mathbb{R}}(B) = 4$) we have $\text{I\!I} = 0$ along $\partial \Sigma$, so that $u(\partial \Sigma)$ is a geodesic in $\partial B$.  %We now briefly comment on the holomorphic case.

Finally, we briefly comment on the holomorphic case.  Suppose $u \colon \Sigma^2 \to B^4$ is a free-boundary holomorphic disk. Let $v \colon \Sigma \to B$ denote a holomorphic, totally-geodesic embedding of a disk as a free-boundary surface (so that $v(\Sigma)$ is a subset of $\mathbb{CP}^1$, $\mathbb{C}^1$, or $\mathbb{CH}^1$, depending on the curvature of the target).  After a holomorphic isometry, we can assume that $v$ and $u$ intersect at a point in the boundary.   Both $u(\partial \Sigma)$ and $v(\partial \Sigma)$ are integral curves of the Reeb field of $S$, so $u(\partial \Sigma) = v(\partial \Sigma)$.  By holomorphicity, it follows that $u = v$ on $\Sigma$, so $u(\Sigma)$ is totally geodesic.
 \end{rmk}

\bibliographystyle{plain}
\bibliography{FBRef}

\Addresses

\end{document}